\documentclass[a4paper,12pt]{amsart}
\usepackage[T1]{fontenc}    %sajat
\usepackage[utf8]{inputenc}%sajat
\usepackage{enumerate}
\usepackage{amssymb}

\def\eps{\varepsilon }

\def\RR{\mathbb R}

\def\ZZ{\mathbb Z}
\def\RRR{\overline{\mathbb R}}
\def\pp{\mathcal P}
\def\qq{\mathcal Q}
\def\mm{\mathcal M}
\def\mmm{\overline{\mathcal M}}
\def\aa{\mathcal A}
\newcommand{\set}[1]{\left\lbrace #1\right\rbrace}%set
\providecommand{\abs}[1]{\left\lvert#1\right\rvert}
\providecommand{\norm}[1]{\left\lVert#1\right\rVert}
\newcommand{\remove}[1]{ }
\newcommand{\qtq}[1]{\quad\text{#1}\quad}%\quad\text{and}\quad

\DeclareMathOperator{\med}{med}
\DeclareMathOperator{\sign}{sign}

\newtheorem{theorem}{Theorem}[section]
\newtheorem{proposition}[theorem]{Proposition}
\newtheorem{lemma}[theorem]{Lemma}
\newtheorem{corollary}[theorem]{Corollary}
\newtheorem*{LA}{Lemma A}
\newtheorem*{LB}{Lemma B}
\newtheorem*{LC}{Lemma C}

\theoremstyle{remark}
\newtheorem*{remark}{Remark}
\remove{}
\newtheorem*{example}{Example}
\newtheorem*{examples}{Examples}

\numberwithin{equation}{section}

\begin{document}
\title[Lebesgue integral]{A simplified construction of the Lebesgue integral}
\remove{\author[V. Komornik]{Vilmos Komornik} 
\address{16 rue de Copenhague\\
         67000 Strasbourg\\
         France}
\email{vilmos.komornik@gmail.com}}
\author{Vilmos Komornik} 
\address{Département de mathématique\\
         Université de Strasbourg\\
         7 rue René Descartes\\
         67084 Strasbourg Cedex, France}
\email{komornik@math.unistra.fr}
%\subjclass{}
%\curraddr{}
\date{Version 2018-05-18}
%\dedicatory{}

\begin{abstract}
We present a modification of Riesz's construction of the Lebesgue integral, leading directly to finite or infinite integrals, at the same time simplifying the proofs.
\end{abstract}
\maketitle

\section{Introduction}\label{s1}

Among the many approaches to the Lebesgue integral that of Riesz \cite{Rie1949B16,Rie1949B18,Riesz-Nagy-1952} is probably the shortest and most elementary. 
As Daniell's abstract method \cite{Daniell-1917}, it is motivated by the research of weak sufficient conditions ensuring the relation $\int f_n\ dx\to\int f\ dx$ for pointwise convergent sequences $f_n\to f$.
Based on two elementary lemmas concerning monotone sequences of step functions, the functions having a finite Lebesgue integral are constructed in two steps.
It is completed by introducing measurable functions having infinite integrals.

This theory has regained popularity recently; see, e.g., Bear \cite{Bear-2002}, Chae \cite{Chae-1995}, Johnston \cite{Johnston-2015}, Roselli \cite{Roselli-2001}.

We modify this method by defining directly the functions having a finite or infinite integral.
Technically we postpone the use of the crucial ``Lemma B'' of Riesz  to the second stage of the construction.
This leads to even shorter and more transparent proofs, and yields the Lebesgue measure at once as a byproduct.
In our approach the fundamental lemmas have a symmetric use: just as Lemma A justifies the correctness of the first step of the extension of the integral, Lemma B plays the same role for the second step.

For clarity the theory is presented in Sections \ref{s2}--\ref{s6} for functions defined on the real line. 
In Sections \ref{s7}--\ref{s8} we extend the results to arbitrary measure spaces and we investigate product measures.
In the final Section \ref{s9} we explain that various difficulties and counter-examples disappear if we return to Riesz's natural definition of measurability and to Fr\'echet's $\sigma$-ring framework. 
We omit in the main text some proofs that remain the same as in the classical setting;   for the convenience of the reader they are recalled in an Appendix.

The author is grateful to L\'aszl\'o Cz\'ach for stimulating discussions and for indicating him the last examples of the paper.

\section{Integral of step functions}\label{s2}

As usual, $N\subset\RR$ is called a \emph{null set} if for each $\eps>0$ it has a countable cover by intervals of total length $<\eps$.
We say that a property holds \emph{almost everywhere} (a.e.) if it holds outside a null set. 

We identify two functions if they are equal almost everywhere, i.e., we are working with equivalence classes of functions.
Accordingly, we write $f=g$, $f\le g$, $f\ge g$, $f_n\to f$, $f_n\nearrow f$ or $f_n\searrow f$ if these relations hold a.e.
(Sometimes we keep the words ``a.e.'' for clarity.)

We denote by $\chi_A$ the characteristic function of the set $A$.
By a \emph{step function} we mean a function having a representation of the form
\begin{equation*}
\varphi=\sum_{i=1}^mc_i\chi_{I_i}
\end{equation*}
with \emph{bounded intervals} $I_i$ and real numbers $c_i$.
We define its \emph{integral} by the formula 
\begin{equation*}
\int\varphi\ dx:=\sum_{i=1}^mc_i\abs{I_i}
\end{equation*}
where $\abs{I_i}$ denotes the length of $I_i$.

\begin{proposition}\label{p21}\mbox{}
\begin{enumerate}[\upshape (i)]
\item The step functions form a vector lattice $R_0$.
\item $\int\varphi\ dx$ does not depend on the particular representation of $\varphi$.
\item The integral is a positive linear functional on $R_0$.
\end{enumerate}
\end{proposition}

\begin{proof}
The proof is elementary and well-known if we consider individual functions and not equivalence classes.
In the latter case (ii) follows from Borel's theorem stating that the non-degen\-erate intervals are \emph{not} null sets.
\end{proof}

In the next two sections we extend the integral in two steps to larger function classes $R_1$  and $R_2$.
We recall without proof two elementary results of Riesz: they will ensure the correctness of these extensions to $R_1$  and $R_2$, respectively.\footnote{For the convenience of the reader we present the missing classical proofs of Sections \ref{s2}--\ref{s8} in an appendix at the end of the paper.}

\begin{LA}\label{lA}
If $(\varphi _n)\subset R_0$ and $\varphi _n(x)\searrow 0$, then $\int \varphi _n\ dx\searrow 0$.
\end{LA}

\begin{LB}\label{lB}
If $(\varphi _n)\subset R_0$, $\varphi _n(x)\nearrow f$ and $\sup\int\varphi_n\ dx<\infty$, then $f$ is finite a.e.
\end{LB}

\section{The $R_1$ function class}\label{s3}

A function $f:\RR\to\RRR$ is said to belong to $R_1$ if there exists a sequence $(\varphi _n)\subset R_0$ satisfying $\varphi_n\nearrow f$, and we set
\begin{equation*}
\int f\ dx:=\lim \int \varphi _n\ dx.
\end{equation*}
Proposition \ref{p31} (i), (ii) below will show that the definition of the integral is correct.

Riesz used instead of $R_1$ the smaller class 
\begin{equation*}
C_1:=\set{f\in R_1\ :\ \int f\ dx\text{ is finite}}.
\end{equation*}
The elements of $R_1$ may have infinite integrals, and they may even take the value $\infty$ on non-null sets.
Nevertheless, we show that the usual properties of $C_1$ continue to hold in $R_1$ by essentially the same proofs.

\begin{proposition}\label{p31}\mbox{}
\begin{enumerate}[\upshape (i)]
\item $\int f\ dx$ does not depend on the particular choice of $(\varphi _n)$.
\item The integral on $R_1$ is an extension of the integral on $R_0$.
\item If $f,g\in R_1$ and $f\le g$, then $\int f\ dx\le \int g\ dx$.
\item If $f,g\in R_1$ and $c$ is a \emph{nonnegative} real number, then\footnote{We use the conventions $0\cdot(\pm\infty):=0$.}
\begin{equation}\label{31}
cf,\quad  f+g,\quad  \min\set{f,g}\qtq{and}\max\set{f,g}
\end{equation}
also belong to $R_1$, and
\begin{equation*}
\int cf\ dx=c\int f\ dx,\quad 
\int f+g\ dx=\int f\ dx+\int g\ dx.
\end{equation*}
\end{enumerate}
\end{proposition}

\begin{proof}
(i) and (iii) follow from Lemma \ref{l32} below. 
(ii) follows from (i) by using constant sequences of step functions.
For (iv) we observe that if  $(\varphi_n), (\psi_n)\subset R_0$, $\varphi_n\nearrow f$ and $\psi_n\nearrow g$, then  the sequences
\begin{equation*}
(c\varphi_n),\quad 
(\varphi_n+\psi_n),\quad 
(\min\set{\varphi_n,\psi_n})\qtq{and}
(\max\set{\varphi_n,\psi_n})
\end{equation*} 
are non-decreasing, belong to $R_0$, and tend to the sequences in \eqref{31}.
The equalities follow from the linearity of the integral on $R_0$.
\end{proof}

\begin{lemma}\label{l32}
Let $(\varphi_n), (\psi_n)\subset R_0$ and $\varphi_n\nearrow f$, $\psi_n\nearrow g$.
If $f\le g$, then 
\begin{equation*}
\lim\int \varphi _n\ dx\le\lim\int \psi _n\ dx.
\end{equation*}
\end{lemma}

\begin{proof}
It suffices to show for each fixed $m$ the inequality 
\begin{equation*}
\int \varphi _m\ dx\le \lim_{n\to\infty} \int \psi _n\ dx.
\end{equation*}
Since the left side is finite,\footnote{We did not need this observation in the classical framework.} 
this is equivalent to the inequality
\begin{equation*}
\lim_{n\to\infty}\int \varphi _m-\psi _n \ dx\le 0.
\end{equation*}
Applying Lemma A to the sequence of functions
\begin{equation*}
(\varphi _m-\psi _n)^+:=\max\set{\varphi _m-\psi _n,0}\searrow 0\qtq{as} n\to\infty
\end{equation*}
we obtain the still stronger relation
\begin{equation*}
\lim_{n\to\infty}\int (\varphi _m-\psi _n)^+ \ dx\le 0.\qedhere
\end{equation*}
\end{proof}

The class $R_1$ is stable for the process used in its definition:

\begin{proposition}\label{p33}
If $(f_n)\subset R_1$ and $f_n\nearrow f$, then $f\in R_1$ and 
\begin{equation*}
\int f_n\ dx\nearrow \int f\ dx.
\end{equation*}
\end{proposition}

\begin{proof}
Fix for each  $n$ a sequence $(\varphi _{n,k})\subset R_0$ satisfying $\varphi _{n,k}\nearrow f_n$.
Then the formula
\begin{equation*}
\varphi _k:=\sup_{n,i\le k}\varphi _{n,i}
\end{equation*}
defines a non-decreasing sequence  of step functions.

For each fixed $n$, we have $\varphi_{n,k}\le\varphi_k\le f$ for all $k\ge n$; letting $k\to\infty$ we conclude that $f_n\le \lim\varphi_k \le f$ for each $n$.
Since $f_n\to f$ a.e., we conclude that $\varphi_k\nearrow f$ a.e.
Therefore $f\in R_1$ and $\int\varphi_k\ dx\to\int f\ dx$.

Since $\varphi_n\le f_n\le f$ for all $n$, integrating and then letting $n\to\infty$ we obtain that 
\begin{equation*}
\lim\int f_n\ dx=\lim\int\varphi_n\ dx=\int f\ dx.\qedhere
\end{equation*}
\end{proof}

\section{The $R_2$ function class and the Lebesgue measure}\label{s4}

It is natural to define $\int -f\ dx:=-\int f\ dx$ if $f\in R_1$. 
More generally, we write $f_1-f_2\in R_2$ if $f_1, f_2\in R_1$ and $\int f_1\ dx-\int f_2\ dx$ is well defined, and for $f=f_1-f_2\in R_2$ we set
\begin{equation*}
\int f\ dx:=\int f_1\ dx-\int f_2\ dx.
\end{equation*}
The function class $R_2$ is well defined: if 
$f_1, f_2\in R_1$, then $\int f_1\ dx-\int f_2\ dx$ is defined if and only if at least one of the two integrals is finite.
Then at least one of the functions $f_1, f_2$ is finite a.e. by Lemma B, so that the function $f_1-f_2$ is defined a.e.

Proposition \ref{p41} (i), (ii) below will show that the definition of the integral is also correct.

\begin{remark}
We may assume in the definition of $R_2$ that $f_1, f_2\ge 0$: choose for $i=1,2$ a step function $\varphi_i$ satisfying $\varphi_i\le f_i$, and change $f_i$ to $f_i-\max\set{\varphi_1,\varphi_2}$.
\end{remark}

Riesz used instead of $R_2$ the smaller class 
\begin{equation*}
C_2:=\set{f\in R_2\ :\ \int f\ dx\text{ is finite}}.
\end{equation*}
The elements of $R_2$ may have infinite integrals, and they are not necessarily finite a.e.
Nevertheless, most properties of the integral on $C_2$ remain valid on $R_2$.

Our first result readily follows from Proposition \ref{p31} for the original class $C_2$; for the extended class $R_2$  some new arguments are needed:

\begin{proposition}\label{p41}\mbox{}
\begin{enumerate}[\upshape (i)]
\item $\int f\ dx$ does not depend on the particular choice of  $f_1$ and $f_2$.
\item The integral on $R_2$ is an extension of the integral on $R_1$.
\item If $f,g\in R_2$ and $f\le g$, then $\int f\ dx\le \int g\ dx$.
\item If $f\in R_2$ and $c\in\RR$, then $cf\in R_2$ and $\int cf\ dx=c\int f\ dx$.
\item If $f,g\in R_2$ and $\int f\ dx+\int g\ dx$ is well defined, then $f+g\in R_2$ and
\begin{equation*}
\int f+g\ dx=\int f\ dx+\int g\ dx.
\end{equation*}
\item If $f,g\in R_2$, then $\max\set{f,g}$ and $\min\set{f,g}$ also belong to $R_2$.
\end{enumerate}
\end{proposition}

\begin{proof}
For (i) and (iii) we have to show that if $f=f_1-f_2$ and $g=g_1-g_2$ as in the definition of $f,g\in R_2$, and $f\le g$, then
\begin{equation}\label{41}
\int f_1\ dx-\int f_2\ dx\le \int g_1\ dx-\int g_2\ dx.
\end{equation}
The inequality is obvious if $\int f_2\ dx=\infty$.
Henceforth we assume that $\int f_2\ dx$ is finite.

If $\int g_2\ dx$ is also finite, then our assumption $f_1-f_2\le g_1-g_2$ implies that $f_1+g_2\le g_1+f_2$ because $f_2$ and $g_2$ are finite a.e. by Lemma B, and then $\int f_1\ dx+\int g_2\ dx\le \int g_1\ dx+\int f_2\ dx$ by Proposition \ref{p31} (iv).
Since $\int g_1\ dx$ and $\int g_2\ dx$ are finite, hence \eqref{41} follows.

If $\int g_2\ dx=\infty$, then choose a sequence $(\varphi_n)\subset R_0$ satisfying $\varphi_n\nearrow g_2$. 
We have $f_1-f_2\le g_1-\varphi_n$ for each $n$.
Applying the preceding arguments with $\varphi_n$ in place of $g_2$ we get 
\begin{equation*}
\int f_1\ dx-\int f_2\ dx\le \int g_1\ dx-\int \varphi_n\ dx,
\end{equation*} 
and \eqref{41} follows by letting $n\to\infty$.
\medskip 

(ii) follows from (i) by choosing $f_1:=f$ and $f_2:=0$ if $f\in R_1$.
\medskip 

(iv) is obvious.
\medskip 

(v) Write $f=f_1-f_2$ and $g=g_1-g_2$ with $f_1, f_2, g_1, g_2\in R_1$ as in the definition of $f, g\in R_2$.
If both integrals $\int f_2\ dx$ and $\int g_2\ dx$ are finite, then Proposition \ref{p31} (iv) shows that $f_1+g_1, f_2+g_2\in R_1$, and $\int f_2+g_2\ dx=\int f_2\ dx+\int g_2\ dx$  is finite.
Therefore
\begin{equation*}
f+g=(f_1+g_1)-(f_2+g_2)
\end{equation*}  
belongs to $R_2$, and using Proposition \ref{p31} (iv) again we obtain that 
\begin{align*}
\int f+g\ dx
&=\int f_1+g_1\ dx-\int f_2+g_2\ dx\\
&=\left(\int f_1\ dx+\int g_1\ dx\right)-\left(\int f_2\ dx+\int g_2\ dx\right)\\
&=\left(\int f_1\ dx-\int f_2\ dx\right)+\left(\int g_1\ dx-\int g_2\ dx\right)\\
&=\int f\ dx+\int g\ dx.
\end{align*}

If one of the integrals $\int f_2\ dx$ and $\int g_2\ dx$ is infinite, then both integrals $\int f_1\ dx$ and $\int g_1\ dx$ are finite by the definition of $f\in R_2$ and by our assumption that $\int f\ dx+\int g\ dx$ is well defined, and we may repeat the above proof.
\medskip 

(vi) By symmetry we consider only the case of maximum. 
We have to show that if $f=f_1-f_2$ and $g=g_1-g_2$ as in the definition of $f,g\in R_2$, then $\max\set{f_1-f_2,g_1-g_2}\in R_2$.
If $f_1, g_1\in R_0$, then
\begin{align*}
\max\set{f_1-f_2,g_1-g_2}
&=(f_1+g_1)+\max\set{-g_1-f_2,-f_1-g_2}\\
&=(f_1+g_1)-\min\set{g_1+f_2,f_1+g_2}
\end{align*}
is the difference of two elements of $R_1$. 
Furthermore,  $\int f_1+g_1\ dx$ is finite because $f_1+g_1\in R_0$, so that the difference belongs to $R_2$.

In the general case we choose two sequences $(\varphi_n), (\psi_n)\subset R_0$ satisfying $\varphi_n\nearrow f_1$ and $\psi_n\nearrow g_1$.
We have $\max\set{\varphi_n-f_2,\psi_n-g_2}\in R_2$ by the preceding arguments, and
\begin{equation*}
\max\set{\varphi_n-f_2,\psi_n-g_2}\nearrow \max\set{f_1-f_2,g_1-g_2}.
\end{equation*}
Since
\begin{equation*}
\int \max\set{\varphi_n-f_2,\psi_n-g_2}\ dx>-\infty
\end{equation*}
for all $n$, because at least one of the integrals $\int f_2\ dx$ and $\int g_2\ dx$ are finite, we may conclude by applying Theorem \ref{t42} below.
\end{proof}

The class $R_2$ has a similar invariance property as $R_1$:

\begin{theorem}[Generalized Beppo Levi theorem]\label{t42}
Let $(f_n)\subset R_2$ and $f_n\nearrow f$. 
If $\int f_n\ dx>-\infty$ for at least one $n$, then $f\in R_2$ and 
\begin{equation}\label{43}
\int f_n\ dx\nearrow \int f\ dx.
\end{equation}
\end{theorem}

\begin{examples}
The functions
\begin{equation*}
f_n:=-\chi_{(-\infty,0)}+\chi_{(0,n)}
\qtq{and} 
f_n:=-\chi_{(n,\infty)}
\end{equation*}
show that the assumption $\lim\int f_n\ dx>-\infty$ cannot be omitted.
Indeed, in the first case  $f=\sign\notin R_2$; in the second case $f=0\in R_2$, but \eqref{43} fails.
\end{examples}

The proof below is a simple adaptation (even a simplification) of the usual one for the smaller class $C_2$.
Since this theorem is fundamental in the present construction, we give the proof here for the convenience of the reader.

\begin{proof}[Proof of Theorem \ref{t42}]
By omitting a finite number of initial terms we may assume that $\int f_n\ dx>-\infty$ for all $n$.
Write $f_n=g_n-h_n$ with $g_n, h_n\in R_1$.
Since $\int h_n\ dx$ is finite, there exists $\varphi_n\in R_0$ satisfying $\varphi_n\le h_n$ and $\int h_n-\varphi_n\ dx<2^{-n}$.
Changing $g_n$ and $h_n$ to $g_n-\varphi_n$ and $h_n-\varphi_n$ we may assume that 
\begin{equation*}
f_n=g_n-h_n,\quad
g_n, h_n\in R_1,\quad 
h_n\ge 0\qtq{and} 
\int h_n\ dx\le 2^{-n}
\end{equation*}
for $n=1,2,\ldots .$
Finally, changing $g_n$ and $h_n$ by induction on $n=2,3,\ldots$ to
\begin{equation*}
h_1+\cdots +h_{n-1}+g_n
\qtq{and} 
h_1+\cdots +h_{n-1}+h_n
\end{equation*}
we may assume that 
\begin{equation*}
f_n=g_n-h_n,\quad
g_n, h_n\in R_1,\quad 
h_n\nearrow \qtq{and} 
\int h_n\ dx\le 1
\end{equation*}
for $n=1,2,\ldots .$

Applying Proposition \ref{p33} to the non-decreasing sequences $(h_n)$ and $(g_n)=(f_n+h_n)$ we obtain that $h_n\nearrow h$ and $g_n\nearrow g$ with suitable functions $h,g\in R_1$, and $\int h\ dx\le 1<\infty$.
Hence $f=g-h\in R_2$ and
\begin{equation*}
\int f_n\ dx=\int g_n\ dx-\int h_n\ dx\to \int g\ dx-\int h\ dx=\int f\ dx.\qedhere
\end{equation*}
\end{proof}

Now we may greatly generalize the length of intervals. 
We write $A\in\mm$  if $\chi_A\in R_2$, and we set $\mu(A):=\int\chi_A\ dx$ in this case. 
Theorem \ref{t42} yields at once the 

\begin{theorem}\label{t43}
$\mu$ is a  $\sigma$-finite, complete measure on $\mm$.
\end{theorem}

The following notions are used here.
Given a family $\aa$ of subsets of a set $X$ with $\varnothing\in\aa$, by a \emph{measure} on $\aa$ we mean a nonnegative function $\mu:\aa\to\RRR$ such that $\mu(\varnothing)=0$, and satisfying the $\sigma$-\emph{additivity} relations
\begin{equation*}
\mu(A)=\sum_{n=1}^{\infty}\mu(A_n)
\end{equation*}
whenever $A\in\aa$ is the disjoint union of a sequence $(A_n)\subset\aa$.

This  measure is called $\sigma$-\emph{finite} if each $A\in\aa$ has a countable cover by sets of finite measure, and \emph{complete} if all subsets of a set of measure zero also belong to $\aa$.

We call the elements of $\mm$ \emph{measurable sets}.

\begin{remark}
We will describe the structure of $\mm$ in Proposition \ref{p62} (iv) below.
\end{remark}

\section{The space $L^1$}\label{s5}

We write $f\in L^1$ if $f\in R_2$ and $\int f\ dx$ is finite, i.e.,
\begin{equation*}
L^1:=\set{f\in R_2\ :\ \int f\ dx\text{ is finite}}.
\end{equation*}
Each $f\in L^1$ is finite a.e. by Lemma B.

Observe that $L^1$ coincides with the class $C_2$ of Riesz.
We are going to show that $L^1$ has a simple structure, and the integral on $L^1$ has some remarkable properties.

\begin{proposition}\label{p51}\mbox{} 
\begin{enumerate}[\upshape (i)]
\item $L^1$ is a vector lattice.
\item The integral is a positive linear functional on $L^1$.
\end{enumerate}
\end{proposition}

\begin{proof}
(i) We show that if $f,g\in L^1$ and $c\in\RR$ then 
\begin{equation*}
cf,\quad f+g,\quad \min\set{f,g}\qtq{and}\max\set{f,g}
\end{equation*}
also belong to $L^1$. 
Write $f=f_1-f_2$ and $g=g_1-g_2$ with  $f_1, f_2, g_1, g_2\in R_1$ having finite integrals.
Then the claim follows from the  representations
\begin{align*}
&cf=cf_1-cf_2\qtq{if}c\ge 0,\\
&cf=(-c)f_2-(-c)f_1\qtq{if}c<0,\\
&f+g=(f_1+g_1)-(f_2+g_2),\\
&\min\set{f,g}=\min\set{f_1+g_2,g_1+f_2}-(f_2+g_2),\\
&\max\set{f,g}=\max\set{f_1+g_2,g_1+f_2}-(f_2+g_2),
\end{align*}
because  by Proposition \ref{p31} (iv) the right side of each equality is the difference of two functions from $R_1$ having finite integrals.
\medskip 

(ii) The integral is linear by Proposition \ref{p31} (iv) and by the definition of the integral on $R_2$, and it is monotone by Proposition \ref{p41} (iii).
\end{proof}

Theorem \ref{t42} yields at once the fundamental

\begin{theorem}[Beppo Levi]\label{t52}
If $(f_n)\subset L^1$, $f_n\nearrow f$ and $\sup\int f_n\ dx<\infty$, then $f\in L^1$ and 
\begin{equation*}
\int f_n\ dx\nearrow \int f\ dx.
\end{equation*}
\end{theorem}

\begin{corollary}\label{c53}\mbox{}
\begin{enumerate}[\upshape (i)]
\item The formula $\norm{f}_1:=\int |f|\ dx$ defines a norm on $L^1$.
\item A set $A$ is a null set $\Longleftrightarrow A\in\mm$ and $\mu(A)=0$.
\end{enumerate}
\end{corollary}

\begin{proof}
(i) The only non-trivial property is that if $\int\abs{f}\ dx=0$, then $f=0$.
This follows by applying Theorem \ref{t52} with $f_n:=n\abs{f}$.
\medskip 

(ii) We have to show that $\chi_A=0\Longleftrightarrow\int\chi_A\ dx=0$.
The implication $\Longrightarrow$ follows from the definition of the integral.
The converse implication follows from (i).
\end{proof}

We recall that Theorem \ref{t52} also implies the following two important theorems:

\begin{theorem}[Fatou]\label{t54}
Let $(f_n)\subset L^1$ and $f_n\to f$.
If $f_n\ge 0$ for all $n$, and $\liminf \int f_n\ dx<\infty$, then $f\in L^1$ and
\begin{equation*}
\int f\ dx\le \liminf \int f_n\ dx.
\end{equation*}
\end{theorem}

\begin{theorem}[Lebesgue]\label{t55}
Let $(f_n)\subset L^1$ and $f_n\to f$. 
If there exists $g\in L^1$ such that $\abs{f_n}\le g$ for all $n$, then $f\in L^1$, and $\int f_n\ dx\to\int f\ dx$.
\end{theorem}

\section{Measurable functions}\label{s6}

We may simplify the manipulation of integrable functions by introducing the notion of measurability.
This will also allow us to precise the structure of the family $\mm$ of measurable sets, and to generalize Fatou's theorem for functions having infinite integrals.

Following Riesz we call a function $f:\RR\to\RRR$  \emph{measurable} if there exists a sequence $(\varphi _n)\subset R_0$  such that $\varphi _n\to f$ a.e.

\begin{proposition}\label{p61} \mbox{} 
\begin{enumerate}[\upshape (i)]
\item If $f\in R_2$, then $f$ is measurable.
\item If $f$ and $g$ are measurable, then $|f|$, $fg$, $\max\set{f,g}$ and $\min\set{f,g}$ are also measurable. 
Furthermore, $f/g$ and $f\pm g$ are also measurable whenever they are defined a.e.\end{enumerate}
\end{proposition}

\begin{proof}
(i) Write $f=f_1-f_2$ with $f_1, f_2\in R_1$.
If $(\varphi_n), (\psi_n)\subset R_0$ and $\varphi_n\nearrow f_1$,  $\psi_n\nearrow f_2$, then $(\varphi_n-\psi_n)\subset R_0$, and $\varphi_n-\psi_n\to f$.
\medskip

(ii) If $(\varphi_n)\subset R_0$ and $\varphi_n\to f$, then $(\abs{\varphi_n})\subset R_0$ and $\abs{\varphi_n}\to \abs{f}$.
The proof of the other statements is analogous. 
\end{proof}

The sign function shows that the not every measurable function belongs to $R_2$.
In our next result we collect several useful statements, including a partial converse of Proposition \ref{p61} (i), the description of the family $\mm$ of measurable sets, and a generalization of Fatou's theorem for functions having infinite integrals.
The proofs will rely on Lebesgue's theorem.

We recall that by a $\sigma$-\emph{ring} in $X$ we mean a family $\mm$ of subsets of $X$ containing $\varnothing$, the difference $A\setminus B$ of any two sets $A,B\in \mm$, and the union $\cup A_n$ of any disjoint sequence $(A_n)\subset\mm$.
Then we have $\cup A_n\in\mm$ and $\cap A_n\in\mm$ for any countable sequence $(A_n)\subset\mm$, even if it is not disjoint.

\begin{proposition}\label{p62}\mbox{} 
\begin{enumerate}[\upshape (i)]
\item If $f$ is measurable, $g\in L^1$ and $\abs{f}\le g$, then $f\in L^1$.
\item If $f$ is measurable \emph{and nonnegative}, then $f\in R_2$.
\item The measurable sets form a $\sigma$-ring.
\item If $(f_n)$ is a sequence of measurable functions and $f_n\to f$, then $f$ is measurable.
\item Let $(f_n)$ be a sequence of nonnegative measurable functions. 
If $f_n\to f$, then $f$ is also a nonnegative measurable function, and
\begin{equation}\label{61}
\int f\ dx\le \liminf \int f_n\ dx.
\end{equation}
\end{enumerate}
\end{proposition}
Part (ii) and Proposition \ref{p61} (i) justify the terminology of Section \ref{s4}: a set is measurable if and only if its characteristic function is measurable.

\begin{proof}
(i) If $(\varphi _n)\subset R_0$ and $\varphi _n\to f$, then the functions\footnote{Here $\med\set{x,y,z}$ denotes the middle number among $x$, $y$ and $z$. For $x\le z$ it is equal to $\max\set{x,\min\set{y,z}}$.}
\begin{equation*}
f_n:=\med\set{-g,\varphi _n,g}
\end{equation*}
belong to the \emph{lattice} $L^1$, and $f_n\to f$.
Since $|f_n|\le g$ for all $n$, we may conclude by applying Lebesgue's theorem.
\medskip

(ii) Choose a non-decreasing sequence $A_1\subset A_2\subset\cdots$ of sets of finite measure such that $f=0$ outside their union.
The functions 
\begin{equation*}
f_n(x):=\min\set{f(x),n\chi_{A_n}}
\end{equation*}
belong to $L^1$ by (i), and $f_n\nearrow f$. 
We conclude by applying Theorem \ref{t42}.
\medskip 

(iii) We have $\varnothing\in\mm$ because $\chi_{\varnothing}=0\in R_0\subset R_2$.
Theorem \ref{t42} yields that $\cup A_n\in\mm$ for any disjoint sequence $(A_n)\subset\mm$.

It remains to show that if $A,B\in \mm$, then $A\setminus B\in\mm$.
By (ii) it suffices to observe that if $\chi_A, \chi_B$ are measurable, then $\chi_{A\setminus B}=\chi_A-\chi_A\chi_B$ is also measurable.
\medskip 

(iv) Fix a \emph{positive} function $g\in L^1$, and set
\begin{equation*}
h_n:=\frac{gf_n}{g+|f_n|}\qtq{and}  h:=\frac{gf}{g+|f|}.
\end{equation*}
Then $h_n$ is measurable and $|h_n|\le g$, so that $h_n\in L^1$ by (i).
Since $h_n\to h$, by Lebesgue's theorem $h\in L^1$, and hence $h$ is measurable by (i).
Since $f$ and $h$ have the same sign, $|f|h=f|h|$, and hence
\begin{equation*}
f=\frac{gh}{g-|h|}
\end{equation*}
is also measurable by Proposition \ref{p61} (ii).
\medskip 

(v) $f$ is a nonnegative measurable function by (iv), hence all integrals in  \eqref{61} are defined by (ii). 
The relation \eqref{61} is obvious if the right side of \eqref{61} is finite; otherwise it follows from Fatou's theorem.\end{proof}

\section{Generalization to arbitrary measure spaces}\label{s7}

The above construction of the Lebesgue integral may be greatly generalized a follows.
Let $\mu:\pp\to\RR$ be a \emph{finite} measure on a semiring $\pp$ in a set $X$, i.e., on a family of sets $\pp\subset 2^X$ having the following properties:
\begin{itemize}
\item $\varnothing \in \pp$;
\item if $P,Q\in \pp $, then $P\cap Q\in \pp $;
\item if $P,Q\in \pp $, then there is a finite disjoint sequence $P_1,\ldots,P_n$  in $\pp$ such that $P\setminus Q=P_1\cup \dots\cup P_n$.
\end{itemize}

Two simple examples are the length of bounded intervals in $\RR$ and the \emph{counting measure} on the finite subsets of any given set $X$: $\mu(P)$ is the number of elements of $P$.

We say that $N\subset X$ is a \emph{null set} if for each $\eps>0$ there exists a sequence $(P_n)\subset\pp$ satisfying $N\subset\cup P_n$ and $\sum\mu(P_n)<\eps$.
We identify two functions if they are equal a.e., i.e., outside a null set.

We denote by $R_0$ the vector space of \emph{step functions} spanned by the characteristic functions of the sets $P\in\pp$, and we define the integral of step functions by the formula
\begin{equation*}
\int \sum_{i=1}^nc_i\chi_{P_i}\ d\mu:=\sum_{i=1}^nc_i\mu(P_i).
\end{equation*}

Now we may repeat the above construction of the integral and measure, and all theorems of Sections \ref{s2}--\ref{s6} remain valid.
Moreover, only the proof of Proposition \ref{p62} (iv) has to be modified because there are measure spaces  containing no positive measurable functions.
It suffices to use a \emph{nonnegative} function $g\in L^1$ satisfying 
\begin{equation*}
g(x)=0\Longrightarrow f_n(x)=0\qtq{for all}n,
\end{equation*} 
and setting $h_n=h:=0$ when $g=0$.
See, e.g., \cite{Riesz-Nagy-1952} or \cite{Komornik-2016} for details.

\section{The Fubini--Tonelli theorem}\label{s8}

Given two finite measures $\mu:\pp\to\RR$ and $\nu:\qq\to\RR$ where $\pp$ is a semiring in $X$ and $\qq$ is a semiring in $Y$, 
\begin{equation*}
\pp\times\qq:=\set{P\times Q\ :\ P\in\pp\qtq{and}Q\in\qq}
\end{equation*}
is a semiring in the product space $\times Y$, and the formula 
\begin{equation*}
(\mu\times\nu)(P\times Q):=\mu(P)\nu(Q)
\end{equation*}
defines a finite measure on $\pp\times\qq$.

We are going to describe the relationship between the three corresponding integrals.
The notation $R_i(X)$, $R_i(X)$, $R_i(X\times Y)$ will refer to the spaces $R_i$ for the measures $\mu$, $\nu$ and $\mu\times\nu$, respectively, and we will write $dx$, $dy$ and $dx\ dy$ instead of $d\mu$, $d\nu$ and $d(\mu\times\nu)$.

\begin{theorem}[Fubini--Tonelli]\label{t81}
We have 
\begin{equation}\label{81}
\int _{X\times Y}f(x,y)\ dx\ dy=\int _X \Bigl(\int _Y f(x,y)\ dy\Bigr)\ dx
\end{equation}
whenever the left side is defined, i.e., $f\in R_2(X\times Y)$.
\end{theorem}
We emphasize that the integrals in \eqref{81} may be infinite.

\begin{remark}
Under the same assumption we also have 
\begin{equation*}
\int _{X\times Y}f(x,y)\ dx\ dy=\int _Y \Bigl(\int _X f(x,y)\ dx\Bigr)\ dy
\end{equation*}
by symmetry.
\end{remark}

\begin{proof}
We may repeat the proof given in \cite{Riesz-Nagy-1952} for the special case $f\in C_2(X\times Y)=L^1(X\times Y)$; it becomes even simpler because we do not have to check the boundedness of the sequences of integrals. 
\end{proof}

\begin{example}
We  recall that the theorem does not hold under the weaker assumption that both successive integrals exist and are equal. 
For example, let  $\mu=\nu$ be the counting measure on $X=Y:=\ZZ$, and 
\begin{equation*}
f(x,y):= 
\begin{cases}
1&\text{if $x=y+1$,}\\
-1&\text{if $x=y-1$,}\\
0&\text{otherwise}.
\end{cases}
\end{equation*}
Then 
\begin{equation*}
\int _X \Bigl(\int _Y f(x,y)\ dy\Bigr)\ dx=\int _Y \Bigl(\int _X f(x,y)\ dx\Bigr)\ dy=0,
\end{equation*}
but the integral $\int _{X\times Y}f(x,y)\ dx\ dy$ is undefined.
\end{example}

\section{$\sigma$-ring or $\sigma$-algebra?}\label{s9}

In this section we argue in favor of the measurability \emph{\`a la} Riesz, and the $\sigma$-rings \emph{\`a la} Fr\'echet \cite{Frechet-1915}, instead of $\sigma$-algebras, i.e., $\sigma$-rings containing the fundamental set $X$. 

The measurability notion adopted here differs from the one used in most modern texts. 
They coincide if $X$ has a countable cover by sets of finite measure (or equivalently by sets belonging to $\pp$), like the Lebesgue measure in $\RR^n$ and the  probability measures.

Otherwise, like for the counting measure on an uncountable set $X$, the present definition is more restrictive: for example the constant functions are not measurable.

In the latter case it is tempting to adopt a weaker definition, by calling a function $f$ \emph{locally measurable} if $f\chi_P$ is measurable for all $P\in\pp$.
Indeed, we may extend the integral to locally measurable functions $f$ by setting $\int f\ dx:=\infty$ if $f$ is nonnegative and non-measurable, and then setting $\int f\ dx:=\int f_+\ dx-\int f_-\ dx$ whenever the right side is well defined.
This extended integral is still monotone.

However, the Fubini--Tonelli theorem may fail for locally measurable functions:

\begin{example}
Let $\mu$ be the zero measure and $\nu$ the counting measure on the finite subsets of an uncountable set $X$.
Then the characteristic function $f$ of the set
\begin{equation*}
D:=\set{(x,x)\ :\ x\in X}
\end{equation*}
is locally measurable for the product measure,
\begin{equation*}
\int _{X\times X}f(x,y)\ dx\ dy
=\infty
\qtq{and}
\int _X \Bigl(\int _X f(x,y)\ dx\Bigr)\ dy=0.
\end{equation*}
\end{example}

Next we take a closer look of Theorem \ref{t43}. 
It may be shown (see, e.g., \cite{Halmos-1950}, Chapter 13) that the measure $\mu:\mm\to\RR$ is the \emph{only} possible extension of its restriction to the initial semiring $\pp$.

If $X$ is measurable, then $\mm$ is not only a $\sigma$-ring, but also a $\sigma$-\emph{algebra}.
Otherwise we may extend $\mu$ further to a $\sigma$-algebra $\mmm$ by setting $\overline{\mu}(A):=\int\chi(A)\ dx$ whenever the characteristic function of $A$ is \emph{locally} measurable.
However, this extension is not unique in general:

\begin{example}
Consider the zero measure $\mu$ on the semiring $\pp$ of  finite subsets of an uncountable set $X$.
Then $\mmm=2^X$, and
\begin{equation*}
\overline{\mu}(A)=
\begin{cases}
0&\text{if $A$ is countable,}\\
\infty&\text{if $A$ is uncountable.}
\end{cases}
\end{equation*}
But the zero measure on $2^X$ is also an extension of $\mu$.

Moreover, the two measures  already differ on the smallest \emph{$\sigma$-algebra} $\aa $ containing $\pp$, i.e., on the family of countable subsets and their complements.
In fact,  there are infinitely many other extensions of $\mu$ to $\aa $: the formula
\begin{equation*}
\mu_{\alpha}(A)=
\begin{cases}
0&\text{if $A$ is countable,}\\
\alpha&\text{if $X\setminus A$ is countable}
\end{cases}
\end{equation*}
defines a different extension of $\mu$ for each $0\le\alpha\le\infty$.
\end{example}\remove{\footnote{L.\ Czách, private communication, 2005.}}

\section{Appendix}\label{s10}

For the convenience of the reader we reproduce here some known proofs that were admitted in the text.

\begin{proof}[Proof of Lemma A in $\RR$]
We may fix a compact interval $[a,b]$ and a  number $M>0$ such that $\varphi_1$, and hence all $\varphi_n$ are bounded by $M$ and vanish outside $[a,b]$.

For any fixed $\eps>0$ there exists a countable sequence of open intervals of total length $<\eps$ such that outside their union $U$ all $\varphi_n$ are continuous.

Then $\varphi_n\to 0$ uniformly on the compact set $[a,b]\setminus U$ by Dini's theorem, so that $0\le\varphi_n<\eps$ on this set for all sufficiently large $n$.
Then we also have
\begin{equation*}
0\le\int\varphi_n\ dx\le M\eps+\eps(b-a),
\end{equation*}
and the lemma follows by the arbitrariness of $\eps$.
\end{proof}
In the proof of Lemma A for a general measure we may use the $\sigma$-additivity of the measure instead of  the topological (compactness) arguments: see \cite{Riesz-Nagy-1952} or \cite{Komornik-2016}.

\begin{proof}[Proof of Lemma B]
Changing $\varphi_n$ to $\varphi_n-\varphi_1$ we may assume that the functions $\varphi_n$ are nonnegative.
Fix an upper bound $A>0$ of the integrals $\int\varphi_n\ dx$, and set
\begin{equation*}
E_{\eps}:=\set{x\in\RR\ :\ f(x)>\frac{A}{\eps}},\quad \eps>0.
\end{equation*}

Since $\set{x\in\RR\ :\ f(x)=\infty}\subset E_{\eps}$, it suffices to show that each $E_{\eps}$ has a countable cover by intervals of total length $\le\eps$.

Setting $f_0:=0$ and
\begin{equation*}
E_{\eps,n}:=\set{x\in\RR\ :\ f_n(x)>\frac{A}{\eps}\ge f_{n-1}(x)},\quad n=1,2,\ldots,
\end{equation*}
$E_{\eps}$ is the union of the disjoint sets $E_{\eps,n}$ by the monotonicity of $(f_n)$.
Since each $E_{\eps,n}$ is a finite union of disjoint intervals $I_{n,1},\ldots,I_{n,k_n}$, it remains to show that 
\begin{equation*}
\sum_{n=1}^{\infty}\sum_{k=1}^{k_n}\abs{I_{n,k}}\le\eps.
\end{equation*}

By the definition of $E_{\eps,n}$ we have for each $m=1,2,\ldots$  the inequality 
\begin{equation*}
\frac{A}{\eps}\sum_{n=1}^m\sum_{k=1}^{k_n}\abs{I_{n,k}}
\le \sum_{n=1}^m\int_{E_n}\varphi_n\ dx 
\le\int\varphi_m\ dx
\le A,
\end{equation*}
and we conclude by letting $m\to\infty$.
\end{proof}

\begin{proof}[Proof of Fatou's theorem \ref{t54}]
Setting 
\begin{equation*}
h_n:=\inf\set{f_k\ :\ k\ge n}
\end{equation*}
we have $h_n\nearrow f$.
Since $h_n\le f_n$ and hence $\sup\int h_n\ dx<\infty$ by our assumption $\liminf \int f_n\ dx<\infty$, the relation $f\in L^1$ will follow by applying Theorem \ref{t52} if we show that $(h_n)\subset L^1$. 

For each fixed $n$ we set 
\begin{equation*}
h_{n,m}:=\min\set{f_k\ :\ n\le k\le m},\quad m=n, n+1,\ldots .
\end{equation*}
Then $(h_{n,m})_{m=n}^{\infty}\subset L^1$ by Proposition \ref{p51} (i).
Since $h_{n,m}\searrow h_n$, and $\int h_{n,m}\ dx\ge 0$ for all $m$, we may apply Theorem \ref{t52} to get $-h_n\in L^1$, and hence $h_n\in L^1$.

Finally, applying  Theorem \ref{t42} and using the relations $h_n\le f_n$ we get the required estimate:
\begin{equation*}
\int f\ dx=\lim\int h_n\ dx=\liminf\int h_n\ dx\le \liminf \int f_n\ dx.\qedhere
\end{equation*} 
\end{proof}

\begin{proof}[Proof of Lebesgue's theorem \ref{t55}]
Applying Fatou's Theorem to the sequences $(g-f_n)$ and $(g+f_n)$ we obtain that $g-f, g+f\in L^1$ and 
\begin{align*}
\int g-f\ dx&\le \liminf \int g-f_n\ dx=\int g\ dx-\limsup \int f_n\ dx,\\
\int g+f\ dx&\le \liminf \int g+f_n\ dx=\int g\ dx+\liminf \int f_n\ dx.
\end{align*}
Since $L^1$ is a vector space, hence $f\in L^1$, and
\begin{equation*}
\limsup\int f_n\ dx\le\int f\ dx\le \liminf \int f_n\ dx.\qedhere
\end{equation*}
\end{proof}

Let us give finally the complete proof of the Fubini--Tonelli theorem.
First we recall from \cite{Riesz-Nagy-1952} an elementary lemma clarifying the relationship between one- and two-dimensional null sets:

\begin{LC}\label{LC}
If $E$  is a null set in $X\times Y$, then the ``sections''
\begin{equation*}
\set{y\in Y \ :\ (x,y)\in E}
\end{equation*}
of $E$ are null sets in $Y$ for almost every  $x\in X$.
\end{LC}

\begin{proof}
Fix a sequence of ``rectangles'' $R_n=P_n\times Q_n$ in $\pp\times \qq$, covering each point of $E$ infinitely many times, and satisfying
\begin{equation*}
\sum_{n=1}^{\infty} (\mu\times\nu)(R_n)<\infty.
\end{equation*}
We may get such a sequence by applying the definition of null sets with $\eps=2^{-n}$ for $n=1,2,\ldots,$ and then combining the resulting covers into one sequence.

By the definition of the integral of  step functions we have
\begin{equation*}
(\mu\times\nu)(R_n)=\int_{X\times Y}\chi_{R_n}(x,y)\ dx\ dy=\int_X\Bigl(\int_Y\chi_{R_n}(x,y)\ dy\Bigr)\ dx
\end{equation*}
(their common value is $\mu(P_n)\nu(Q_n)$), so that the series
\begin{equation*}
\sum_{n=1}^{\infty}  \int _X\Bigl(\int _Y\chi_{R_n}(x,y)\ dy\Bigr)\ dx
\end{equation*}
is convergent.
Applying the Beppo Levi theorem we obtain that the series
\begin{equation*}
\sum_{n=1}^{\infty}  \int _Y\chi_{R_n}(x,y)\ dy
\end{equation*}
is convergent for a.e.\ $x\in X$.
If $x_0$ is such a point, then another application of the Beppo Levi theorem implies that the series
\begin{equation*}
\sum_{n=1}^{\infty}  \chi_{R_n}(x_0,y)
\end{equation*}
is convergent for a.e.\ $y\in Y$.
If $y_0$ is such a point, then $(x_0,y_0)\notin E$, because at the points of $E$ we have $\sum \chi_{R_n}=\infty$.
\end{proof}

\begin{proof}[Proof of Theorem \ref{t81}]
We have to show the following: 
\begin{align}
&f(x,\cdot)\in R_2(Y)\qtq{for almost each}x\in X;\label{82}\\
&\int _Y f(\cdot,y)\ dy\in R_2(X);\label{83} \\
&\text{the two sides of \eqref{81} are equal.}\label{84}
\end{align}

These properties obviously hold if $f$ is the characteristic function of a ``rectangle'' $P\times Q$.
Taking linear combinations we see that they hold for all step functions $f\in R_0(X\times Y)$ as well.

Next let $f\in R_1(X\times Y)$.
Choose a sequence $(\varphi _n)\subset R_0(X\times Y)$ and a null set $E\subset X\times Y$ such that
\begin{equation*}
\varphi _n(x,y)\nearrow f(x,y)\qtq{for each} (x,y)\in (X\times Y)\setminus E;
\end{equation*}
then
\begin{equation}\label{85}
\int _{X\times Y}\varphi _n(x,y)\ dx\ dy\nearrow \int _{X\times Y}f(x,y)\ dx\ dy
\end{equation}
by the definition of the integral.

By Lemma C we may fix a set $X_1\subset X$ such that $X\setminus X_1$ is a null set in $X$, and for each fixed $x\in X_1$, $(\varphi _n(x,\cdot))\subset R_0(Y)$ and
\begin{equation*}
\varphi _n(x,\cdot)\nearrow f(x,\cdot)\qtq{in} Y,
\end{equation*}
implying $f(x,\cdot)\in R_1(Y)$ and the relation
\begin{equation*}
\int_Y\varphi _n(x,y)\ dy\nearrow\int_Y f(x,y)\ dy\qtq{for a.e.}x\in X.
\end{equation*}
Since $(\int_Y\varphi _n(\cdot,y)\ dy)\subset R_0(X)$, this implies $\int_Yf(\cdot,y)\ dy\in R_1(X)$ and the relation
\begin{equation}\label{86}
\int_X\left(\int_Y\varphi _n(x,y)\ dy\right)\ dx\nearrow\int_X\left(\int_Yf(x,y)\ dy\right)\ dx.
\end{equation}
In particular, we have established \eqref{82} and \eqref{83}.
The equality \eqref{84} follows by observing that the left sides of \eqref{85} and \eqref{86} are equal, and that \eqref{84} is already known for step functions.

If $f\in R_1(X\times Y)$ and $\int_{X\times Y}f(x,y)\ dx\ dy<\infty$, then applying Lemma B we also see that $f$ is finite a.e. in $X\times Y$, and $\int_Yf(x,y)\ dy$ is finite for every $x\in X_1$.

Finally, if $f\in R_2(X\times Y)$, then writing $f=f_1-f_2$ with $f_1, f_2\in R_1(X\times Y)$ we have the required equality for $f_1$ and $f_2$ in place of $f$.
Since one of the integrals $\int_{X\times Y}f_1(x,y)\ dx\ dy$ and $\int_{X\times Y}f_2(x,y)\ dx\ dy$ is finite, by the preceding paragraph we may take the difference of these equalities to obtain the required identity for $f$.
\end{proof}

We end the appendix by recalling a not too well-known  example of Weir \cite[p. 43]{Weir-1973} (see also Johnston \cite[pp. 54--55]{Johnston-2015}).
The strict inclusion $R_1\subsetneq R_2$ follows from the lower boundedness of the elements of $R_1$. 
The converse does not hold: a nonnegative element of $R_2$ does not necessarily belong to $R_1$.
To see this we enumerate the rational numbers in $(0,1)$ into a sequence $(r_n)$ and we set 
\begin{equation*}
S:=(0,1)\cap\left(\cup_{n=1}^{\infty}(r_n-2^{-n-3},r_n+2^{-n-3})\right).
\end{equation*}
Then $S\subset (0,1)$ and $0<\mu(S)<1$.

Now $f:=\chi_{(0,1)}$ and $g:=\chi_S$ belong to $R_1$, so that $f-g\in R_2$, and $f-g\ge 0$.
However, $f-g\notin R_1$ because $\int f-g\ dx>0$, and $\varphi\le 0$ for every step function satisfying $\varphi\le f-g$.

\end{document}